\documentclass[12pt,a4paper]{amsart}

\newtheorem{theo+}              {Theorem}           [section]
\newtheorem{prop+}  [theo+]     {Proposition}
\newtheorem{coro+}  [theo+]     {Corollary}
\newtheorem{lemm+}  [theo+]     {Lemma}
\newtheorem{exam+}  [theo+]     {Example}
\newtheorem{rema+}  [theo+]     {Remark}
\newtheorem{defi+}  [theo+]     {Definition}

\newenvironment{theorem}{\begin{theo+}}{\end{theo+}}
\newenvironment{proposition}{\begin{prop+}}{\end{prop+}}
\newenvironment{corollary}{\begin{coro+}}{\end{coro+}}
\newenvironment{lemma}{\begin{lemm+}}{\end{lemm+}}

\usepackage{amsthm}
\theoremstyle{plain} \theoremstyle{remark}
\newtheorem{remark}{Remark}
\newtheorem{example}{Example}

\def\E{/\kern-1.0em \equiv }

\author{Ze-Ping Wang }
\author{Ye-Lin Ou$^{*}$ }
\author{Han-Chun Yang$^{**}$ }
\address{Department of Mathematics, \newline\indent Yunnan University,\newline\indent
Kunming 650091, P. R. China
\newline\indent E-mail:zpwzpw2012@126.com \;(Wang)\\\newline\indent E-mail: hyang@ynu.edu.cn\; (H. Yang)\\\newline\indent
\newline\indent Department of
Mathematics,\newline\indent Texas A $\&$ M
University-Commerce,\newline\indent Commerce, TX 75429
USA.\newline\indent E-mail:yelin.ou@tamuc.edu\; (Ou)}
\thanks{* A part of the work in this paper was done while Ye-Lin Ou was visiting Department of Mathematics at Yunnan University in July, 2014. He was grateful to the department and Prof. Han-Chun Yang for the hospitality he received during the visit.\\
** Research supported by the NSF of China (11361073)}

\begin{document}

\title[Biharmonic and $f$-biharmonic maps from a 2-sphere]{Biharmonic and $f$-biharmonic maps from a 2-sphere}

\subjclass{58E20, 53C43} \keywords{Biharmonic maps, $f$-biharmonic maps, Conformal chnage of metrics, 2-sphere, Riemann sphere.}
\date{01/12/2015}
\maketitle

\section*{Abstract}
\begin{quote}
{\footnotesize We study biharmonic maps and $f$-biharmonic maps from a round sphere $(S^2, g_0)$, the latter maps are equivalent to biharmonic maps from Riemann spheres $(S^2, f^{-1}g_0)$. We proved that for rotationally symmetric maps between rotationally symmetric spaces, both biharmonicity and $f$-biharmonicity reduce to a 2nd order linear ordinary differential equation. As applications, we give a method to produce biharmonic maps and $f$-biharmonic maps from given biharmonic maps and we construct many examples of biharmonic and $f$-biharmonic maps from a round sphere $S^2$ and between two round spheres. Our examples include non-conformal proper biharmonic maps $(S^2, f^{-1}g_0)\longrightarrow S^2$ and $(S^2, f^{-1}g_0)\longrightarrow S^n$, or non-conformal $f$-biharmonic maps $(S^2,
g_0)\longrightarrow S^2$ and $(S^2,g_0)\longrightarrow S^n$ from round sphere with two singular points.}
\end{quote}
\section{introduction}
All objects including manifolds, tensor fields, and maps in this paper are assumed to be smooth if they are not stated otherwise.\\

{\em Harmonic maps} are maps between Riemannian manifolds $\phi: (M,g)\longrightarrow (N,h)$ whose tension fields vanish identically, i.e.,
\begin{equation}\label{fhm}
\tau(\phi) \equiv {\rm Tr}_g\nabla\,d \phi=0,
\end{equation}
where $\tau(\phi)={\rm Tr}_g\nabla\,d \phi$ is called the tension field of
the map $\phi$.\\

{\em Biharmonic maps and $f$-biharmonic maps} are maps $\phi: (M, g)\longrightarrow (N,h)$ between Riemannian manifolds which are critical points of, respectively, the bienergy and the $f$-bienergy functionals
\begin{equation}\notag
E_{2}(\phi)=\frac{1}{2}\int_\Omega |\tau(\phi)|^2v_g,\;\;\;E_{2,f}(\phi)=\frac{1}{2}\int_\Omega f\,|\tau(\phi)|^2v_g,
\end{equation}
for all variations over any compact domain $\Omega$ of $ M$. The
Euler-Lagrange equations of the two functionals give, respectively, {\em the biharmonic map
equation} (\cite{Ji1}):
\begin{equation}\label{BTF}
\tau_{2}(\phi):={\rm
Trace}_{g}(\nabla^{\phi}\nabla^{\phi}-\nabla^{\phi}_{\nabla^{M}})\tau(\phi)
- {\rm Trace}_{g} R^{N}({\rm d}\phi, \tau(\phi)){\rm d}\phi =0,
\end{equation}
where $R^{N}$ denotes the curvature operator of $(N, h)$ defined by
$$R^{N}(X,Y)Z=
[\nabla^{N}_{X},\nabla^{N}_{Y}]Z-\nabla^{N}_{[X,Y]}Z,$$ and the {\em $f$-biharmonic map equation} ( See \cite{Lu} and also \cite{Ou6}):
\begin{equation}\label{FBH} 
\tau_{2,f} (\phi)\equiv f\tau_2(\phi)+(\Delta f)\tau(\phi)+2\nabla^{\phi}_{{\rm grad}\, f}\tau(\phi)=0,
\end{equation}
where $\tau(\phi)$ and $\tau_2( \phi)$ are the tension and the bitension fields of
$\phi$ respectively.\\

Clearly, we have the following relationships among these different types of harmonicities of maps:
$$\{{\rm \bf Harmonic\; maps}\}\subset \{{\rm \bf Biharmonic \;maps} \}\subset \{{\rm \bf f-Biharmonic \;maps}\}.$$
A biharmonic map which is not harmonic is called {\em a proper biharmonic map}, similarly, we call an $f$-biharmonic map which is neither harmonic nor biharmonic a {\em proper $f$-biharmonic map}.\\

In this paper, we study biharmonic maps and $f$-biharmonic maps from a $2$-sphere. An interesting fact is that an $f$-biharmonic map from the round sphere $(S^2, g_0)$ is exactly equivalent to a biharmonic map from the Riemann sphere $(S^2, f^{-1}g_0)$ conformal to the standard one by a theorem proved in a recent work \cite{Ou6} of the second named author. The paper is a continuation of our previous work \cite{WOY} with a goal to understand which parts of the rich theory and applications of harmonic maps from a $2$-sphere can be generalized to the case of biharmonic maps. In particular, it would be interesting to know 
\begin{itemize}
\item if there is a proper biharmonic map $\varphi :S^2\longrightarrow (N^n, h)$ that is not a weakly conformal immersion. A well-known theorem of Sacks-Uhlenbeck \cite{SU} shows that any harmonic map $\varphi:S^2\longrightarrow (N^n,h)$ with $n\ge 3$ has to be a conformal branched minimal immersion;
\item if there is a proper biharmonic map representative in any homotopy class of maps $S^2\longrightarrow S^2$. Smith \cite{Sm} proved that harmonic map representative exists in each homotopy class;
\item if there is a proper biharmonic embedding of $S^2$ into $S^3$, equipped with arbitrary metric. A result of Smith in \cite{SM1} shows that it is true for harmonic embeddings of $S^2$ into $S^3$, equipped with arbitrary metric.
\end{itemize} 

So far, the only known example of proper biharmonic maps from $S^2$ is the biharmonic isometric
immersion ( \cite{CMO1}): $S^2(\frac{1}{\sqrt{2}})\longrightarrow S^3$ (or a composition of this with a totally geodesic map from $S^3$ into another manifold (See, e.g., \cite{Ou1})).

 In \cite{WOY}, we proved that there is no proper biharmonic map in the family of rotationally symmetric maps $\varphi:(S^2, dr^2+\sin^2r d\theta^2)\longrightarrow
(S^2, d\rho^2+\sin^2\rho d\phi^2)$ with $\varphi(r,\theta)=(Ar, k\theta)$. Some proper biharmonic maps from $S^2\setminus\{N, S\}$ into an open domain $(\mathbb{R}^2\supset D, d\rho^2+(\rho^2+C_1\rho+C_0)d \phi)$ were also constructed in \cite{WOY}. 

In this paper, we first proved that for rotationally symmetric maps between rotationally symmetric spaces, both biharmonicity and $f$-biharmonicity reduce to a second order linear ordinary differential equation. As applications of this, we give a method to produce biharmonic maps and $f$-biharmonic maps from given biharmonic maps and we construct many examples of proper biharmonic and $f$-biharmonic maps from a round sphere $S^2$ and between two round spheres. Our examples include non-conformal proper biharmonic maps $(S^2, f^{-1}g_0)\longrightarrow S^2$ and $(S^2, f^{-1}g_0)\longrightarrow S^n$, or non-conformal $f$-biharmonic maps $(S^2, g_0)\longrightarrow S^2$ and $(S^2,g_0)\longrightarrow S^n$ from round sphere with two singular points. There examples provide an interesting comparison to Sack-Uhlenbeck's well-know theorem stating that any harmonic map from $S^2$ is conformal immersion away from points where the differential of the map vanishes.

\section{Biharmonicity and $f$-biharmonicity of rotationally symmetric maps}
A striking difference between harmonic maps and biharmonic maps lies in the fact that harmonicity of maps from a $2$-dimensional manifold is invariant under conformal change of the metric on the domain. Although biharmonicity does not enjoy this property, we do have the following interesting link between biharmonicity and $f$-biharmonicity when the domain is a $2$-dimensional manifold.
\begin{theorem}\cite{Ou6}\label{BFB}
A map $\phi: (M^2, g)\longrightarrow (N^n, h)$ is
an $f$-biharmonic map if and only if $\phi:
(M^2, f^{-1}g)\longrightarrow (N^n,h)$ is a biharmonic map.
\end{theorem}
In particular, we have
\begin{corollary}\label{C20}
A map $\phi: (S^2, g_0)\longrightarrow (N^n, h)$ from the standard sphere is
an $f$-biharmonic map if and only if the map $\phi:
(S^2, f^{-1}g)\longrightarrow (N^n,h)$ is a biharmonic map from a Riemann sphere.
\end{corollary}
This is another motivation for us to study $f$-biharmonic maps from a $2$-sphere. 
There are many rotationally symmetric maps between $2$-spheres, for instance, $F:\mathbb{R}^{3}\longrightarrow \mathbb{R}^{3}$ defined by $F(x, y, z)=(x^{2}-y^{2}-z^{2}, \;2xy,\; 2xz)$ defines a smooth map $\varphi: S^2\longrightarrow S^2$, which is a rotationally symmetric map
\begin{eqnarray}\notag
&&\varphi :((0,\pi/2)\times S^{1}, {\rm d}\,r^{2}+\sin ^{2} r\,{\rm d}\,\theta^{2})\longrightarrow ((0,\pi)\times S^{1}, {\rm d}\,\rho^{2}+\sin ^{2} \rho\,{\rm d}\,\phi^{2}),\\ \notag
&&\varphi (r,\theta)=(2r, \theta).
\end{eqnarray}

We know from \cite{WOY} that this map $\varphi$ is neither harmonic nor biharmonic. It would be interesting to know if we can make a conformal change of the domain metric so that the map becomes a biharmonic map with respect to the new metric. This problem amounts exactly to study $f$-biharmonicity of rotationally symmetric maps.\\

The idea of constructing proper biharmonic maps by conformal changes of the domain metrics of a harmonic map from a manifold of dimension $\ge 3$ was first used by Baird and Kamissoko in \cite{BK}. Ou in \cite{Ou2} gave the transformation of the bitension field of a generic map $\phi: (M^m, g)\longrightarrow (N^n, h)$ under conformal change of the domain metric. In particular, the following lemma proved in \cite{Ou2} will be used to derive the $f$-biharmonic equation for rotationally symmetric maps between rotationally symmetric manifolds.
\begin{lemma} ( See \;\cite{Ou2} )\label{c}
Let $\varphi : (M^{2},g)
\longrightarrow (N^{n},h)$ be a map and ${\bar g}=F^{-2}g$ be a
conformal change of the metric $g$. Let $\tau^{2}(\varphi, g)$ and $\tau^{2}(\varphi,{\bar g})$ be the bitension fields of $\varphi$
with respect to the metrics $g$ and ${\bar g}$ respectively. Then,
\begin{eqnarray}\label{tao1}
\tau_{2}(\varphi,{\bar g})&=& F^4\{\tau_{2}(\varphi, g) +2 (\Delta
{\rm ln}F+2\left|{\rm grad\,ln}F\right|^2)\tau(\varphi, g))\\\notag
&& +4\nabla^{ \varphi}_{{\rm grad}\,\ln F}\,\tau(\varphi, g)\}.
\end{eqnarray}
\end{lemma}
Applying Lemma \ref{c} we have
\begin{lemma}\label{sl}
Let $\varphi:(M^2,g=dr^2+\sigma^2(r)d\theta^2)\longrightarrow
(N^2,h=d\rho^2+\lambda^2(\rho)d\phi^2)$ be a rotationally symmetric map with $\varphi(r,\theta)=(\rho(r),k\theta)$. Then, the map $\varphi:(M^2,\bar{g}=f^{-1}g)\longrightarrow
(N^2,h)$ is biharmonic if and only if $f=f(r, \theta)$ solves the system
\begin{equation}\label{Sl1}
\begin{cases}
\Delta (xf)=\frac{k^2(\lambda\lambda'(\rho))'(\rho)}{\sigma^2}\,(xf),\\
\frac{k f_\theta\lambda'(\rho)x}{\sigma^2\lambda}=0,\\
x=\Delta \rho-\frac{k^2\lambda\lambda'(\rho)}{\sigma^2},
\end{cases}
\end{equation}
where $\Delta $ denote the Laplacian on functions taken with respect to the metric $g$. In particular, when $f=f(r)$ depends only on variable $r$, then the map $\varphi:(M^2,\bar{g}=f^{-1}g)\longrightarrow
(N^2,h)$ is biharmonic if and only if $f$ solves the following equation
\begin{equation}\label{Sl01}
\begin{cases}
y''+\frac{\sigma'}{\sigma}y'-\frac{k^2(\lambda\lambda'(\rho))'(\rho)}{\sigma^2}y=0,\\
y=f x,\\
x=\rho''+\frac{\sigma'}{\sigma}\rho'-\frac{k^2\lambda\lambda'(\rho)}{\sigma^2}.
\end{cases}
\end{equation}
\end{lemma}
\begin{proof}
A straightforward computation gives
\begin{equation}\label{Sl2}
\begin{array}{lll}
\tau(\varphi,g)
=(\rho''+\frac{\sigma'}{\sigma}\rho'-\frac{k^2\lambda\lambda'(\rho)}{\sigma^2})\frac{\partial}{\partial\rho},\\
\tau(\varphi,\bar{g})=f\tau(\varphi,g)
=f(\rho''+\frac{\sigma'}{\sigma}\rho'-\frac{k^2\lambda\lambda'(\rho)}{\sigma^2})\frac{\partial}{\partial\rho},
\end{array}
\end{equation}
and
\begin{equation}\label{Sl3}
\begin{array}{lll}
\tau^2(\varphi,g)=(x''+\frac{\sigma'}{\sigma}x'-\frac{k^2(\lambda\lambda'(\rho))'(\rho)}{\sigma^2}x)\frac{\partial}{\partial\rho},\\
\end{array}
\end{equation}
where $x=\rho''+\frac{\sigma'}{\sigma}\rho'-\frac{k^2\lambda\lambda'(\rho)}{\sigma^2}$.
Using Lemma \ref{c}, the bitension fields of $\varphi$
with respect to the metric ${\bar g}$ can be computed as
\begin{equation}\label{Sl4}
\begin{array}{lll}
\tau^2(\varphi,\bar{g})=f^2\{\tau^{2}(\varphi, g) +2 (\Delta
{\rm ln}f^{1/2}+2\left|{\rm grad\,ln} f^{1/2}\right|^2)\tau(\varphi, g))\\
+4\nabla^{ \varphi}_{{\rm grad}\,\ln f^{1/2}}\,\tau(\varphi, g)\}\\
=f^2\left(x''+(\frac{\sigma'}{\sigma}+2(\ln f)_r)x'+(\Delta (\ln f)+|{\rm grad} (\ln f)|^2-\frac{k^2(\lambda\lambda'(\rho))'(\rho)}{\sigma^2})x\right)\frac{\partial}{\partial\rho}\\
+\frac{2k f_\theta\lambda'(\rho)x}{f\sigma^2\lambda}\frac{\partial}{\partial\phi}.
\end{array}
\end{equation}

From this, we conclude that the map $\varphi:(M^2,\bar{g}=f^{-1}g)\longrightarrow
(N^2,h)$ is biharmonic if and only if $f=f(r, \theta)$ solves Equation (\ref{Sl1}).
If $ f=f(r)$ depends only variable $r$, then $f_\theta=0$, and we can check that Equation (\ref{Sl1}) is equivalent to
Equation (\ref{Sl01}). This completes the proof of the lemma.
\end{proof}
Using the relationship between biharmonic maps and $f$-biharmonic maps ( See Theorem \ref{BFB}) we have
\begin{corollary}\label{f-bih}
The rotationally symmetric map $\varphi:(M^2, dr^2+\sigma^2(r)d\theta^2)\longrightarrow
(N^2, d\rho^2+\lambda^2(\rho)d\phi^2)$ with $\varphi(r,\theta)=(\rho(r),k\theta)$
is an $f$-biharmonic map if and only if $ f$ solves Equation (\ref{Sl1}).
In particular, when $f=f(r)$ depends only on the variable $r$, the rotationally symmetric map $\varphi:(M^2, dr^2+\sigma^2(r)d\theta^2)\longrightarrow
(N^2, d\rho^2+\lambda^2(\rho)d\phi^2)$ with $\varphi(r,\theta)=(\rho(r),k\theta)$ is an $f$-biharmonic map if and only if $f$ solves Equation (\ref{Sl01}).
\end{corollary}
\begin{remark}\label{reyj}
Note that if we set $t=\int \frac{{\rm dr}}{\sigma}$ and assume that $x=\rho''+\frac{\sigma'}{\sigma}\rho'-\frac{k^2\lambda\lambda'(\rho)}{\sigma^2}\neq0$, then we can check that the first equation
of (\ref{Sl01}) is equivalent to
\begin{equation}
y''(t)-k^2(\lambda\lambda'(\rho))'(\rho)y(t)=0.
\end{equation}
In this case $f(t) =\frac{y(t)}{x}$ with $t=\int \frac{{\rm dr}}{\sigma}$.
\end{remark}
Our next theorem gives a method to construct proper biharmonic maps from a given proper biharmonic map via a conformal change of the metric on the domain surface.
\begin{theorem}\label{stco}
Let $\varphi:(M^2, g=dr^2+\sigma^2(r)d\theta^2)\longrightarrow
(N^2, d\rho^2+\lambda^2(\rho)d\phi^2)$ with $\varphi(r,\theta)=(\rho(r),k\theta)$ be a proper biharmonic map with $\tau(\varphi)\ne 0$. Then, for constants $C_1, C_2$ with $C_1^2+C_2^2\neq 0$ and $f=C_1+C_2\int(\rho''+\frac{\sigma'}{\sigma}\rho'-\frac{k^2\lambda\lambda'(\rho)}{\sigma^2})^{-2}\sigma^{-1}{\rm d}r$, the rotationally symmetric map $\varphi:(M^2,f^{-1}g)\longrightarrow
(N^2,h)$ is a proper biharmonic map.
\end{theorem}
\begin{proof}
For the rotationally symmetric map $\varphi:(M^2,g=dr^2+\sigma^2(r)d\theta^2)\longrightarrow
(N^2,h=d\rho^2+\lambda^2(\rho)d\phi^2)$ with $\varphi(r,\theta)=(\rho(r),k\theta)$, we know that $\tau(\varphi)=x\frac{\partial}{\partial \rho}$. It follows from Corollary 2.3 in \cite{WOY} that it is a biharmonic map if and only if 
\begin{equation}\label{stco1}
\begin{cases}
x''+\frac{\sigma'}{\sigma}x'-\frac{k^2(\lambda\lambda'(\rho))'(\rho)}{\sigma^2}x=0,\\
x=\rho''+\frac{\sigma'}{\sigma}\rho'-\frac{k^2\lambda\lambda'(\rho)}{\sigma^2},
\end{cases}
\end{equation}
which implies that $x=\rho''+\frac{\sigma'}{\sigma}\rho'-\frac{k^2\lambda\lambda'(\rho)}{\sigma^2}\ne 0$ is a solution of the 2nd order linear DE
\begin{equation}\label{stco2}
z''+\frac{\sigma'}{\sigma}z'-\frac{k^2(\lambda\lambda'(\rho))'(\rho)}{\sigma^2}z=0.
\end{equation}
It follows from the reduction of order theory of second order linear ODEs that the general solution of (\ref{stco2}) can be written as
\begin{equation}\label{stco3}
z=x\left(C_1+C_2\int\frac{1}{x^2}e^{-\int \frac{\sigma'}{\sigma}{\rm d}r}{\rm d}r\right)=x\left(C_1+C_2\int(\rho''+\frac{\sigma'}{\sigma}\rho'-\frac{k^2\lambda\lambda'(\rho)}{\sigma^2})^{-2}\sigma^{-1}{\rm d}r\right),
\end{equation}
where $C_1, C_2$ are constants. On the other hand, using Lemma \ref{sl}, we know that the map $\varphi:(M^2,\bar{g}=f^{-1}(r)g)\longrightarrow
(N^2,h)$ is biharmonic if and only if we have Equation (\ref{Sl01}), which means the function
$y=x(r)f(r)$ is also a solution of (\ref{stco2}). By comparing this to the general solution given in (\ref{stco3}), we have $f=C_1+C_2\int(\rho''+\frac{\sigma'}{\sigma}\rho'-\frac{k^2\lambda\lambda'(\rho)}{\sigma^2})^{-2}\sigma^{-1}{\rm d}r$, from which we obtain the theorem.
\end{proof}

As a straightforward application of Theorem \ref{stco}, we have the following corollary which can be used to construct non-harmonic $f$-biharmonic maps from given proper biharmonic maps.
\begin{corollary}
For any proper biharmonic map $\varphi:(M^2,g=dr^2+\sigma^2(r)d\theta^2)\longrightarrow
(N^2,h=d\rho^2+\lambda^2(\rho)d\phi^2)$ with $\varphi(r,\theta)=(\rho(r),k\theta)$ with $|\tau(\varphi)|\ne 0$, the map $\varphi:(M^2, g)\longrightarrow
(N^2,h)$ is a non-harmonic $f$-biharmonic for $f=C_1+C_2\int(\rho''+\frac{\sigma'}{\sigma}\rho'-\frac{k^2\lambda\lambda'(\rho)}{\sigma^2})^{-2}\sigma^{-1}{\rm d}r$, where $C_1, C_2$ are constants and $C_1^2+C_2^2\neq0$.
\end{corollary}
\begin{theorem}\label{stcoy}
Let $c_1, c_2, C_1, C_2, C_3, C_4, k\neq0, C_0, C$ be constants so that $\lambda^2(\rho)=2C_0\rho+C>0$ and $(c_1^2+c_2^2)(C_1^2+ C_2^2)\neq0$. Then, for $f=c_1+c_2\int \{(C_1\int\sigma^{-1}{\rm d}r+C_2)^{-2}\sigma^{-1}\}{\rm d}r $, the map $\varphi:(M^2,f^{-1}({\rm d}r^2+\sigma^2(r){\rm d}\theta^2))\longrightarrow
(N^2,{\rm d}\rho^2+\lambda^2(\rho){\rm d}\phi^2)$, $\varphi(r,\theta)=(\rho(r), k\theta)$ with $\rho(r)$ defined by $\rho(r)=\int\left\{\frac{\int\left(C_1\sigma(r)\int\frac{{\rm d}r}{\sigma(r)}+C_2\sigma(r)+\frac{k^2C_0}{\sigma(r)}\right){\rm d}r+C_3}{\sigma(r)}\right\}{\rm d}r+C_4$, is proper biharmonic.
\end{theorem}
\begin{proof}
It was proved in the Claim in the proof of Theorem 4.4 in \cite{WOY} that for $C_1^2+C_2^2\neq0$ and $\lambda^2(\rho)=2C_0\rho+C$, the map $\varphi:(M^2,{\rm d}r^2+\sigma^2(r){\rm d}\theta^2)\longrightarrow
(N^2,{\rm d}\rho^2+\lambda^2(\rho){\rm d}\phi^2)$ with $\varphi(r,\theta)=(\rho(r),k\theta)$ and $\rho(r)$ defined by\\ $\rho(r)=\int\left\{\frac{\int\left(C_1\sigma(r)\int\frac{{\rm d}r}{\sigma(r)}+C_2\sigma(r)+\frac{k^2C_0}{\sigma(r)}\right){\rm d}r+C_3}{\sigma(r)}\right\}{\rm d}r+C_4$ is proper biharmonic. Using Theorem \ref{stco} we conclude that for $c_1^2+c_2^2$ and $f=c_1+c_2\int(\rho''+\frac{\sigma'}{\sigma}\rho'-\frac{k^2\lambda\lambda'(\rho)}{\sigma^2})^{-2}\sigma^{-1}{\rm d}r$, the map $\varphi:(M^2,f^{-1}({\rm d}r^2+\sigma^2(r){\rm d}\theta^2))\longrightarrow
(N^2,{\rm d}\rho^2+\lambda^2(\rho){\rm d}\phi^2)$, $\varphi(r,\theta)=(\rho(r), k\theta)$ with $\rho(r)$ defined by $\rho(r)=\int\left\{\frac{\int\left(C_1\sigma(r)\int\frac{{\rm d}r}{\sigma(r)}+C_2\sigma(r)+\frac{k^2C_0}{\sigma(r)}\right){\rm d}r+C_3}{\sigma(r)}\right\}{\rm d}r+C_4$, is proper biharmonic. A further computation shows that $f=c_1+c_2\int \{(C_1\int\sigma^{-1}{\rm d}r+C_2)^{-2}\sigma^{-1}\}{\rm d}r $, from which we obtain the theorem.
\end{proof}
Similarly, using Theorem 4.5 in \cite{WOY} and Theorem \ref{stco} we obtain the following theorem which gives locally defined $f$-biharmonic map from a sphere.
\begin{theorem}\label{stcoy1}
Let $c_1, c_2, C_1, C_2, C_3, C_4, k\neq0, C_0, C$ be constants so that $\lambda^2(\rho)=\rho^2+2C_0\rho+C>0$ and $(c_1^2+c_2^2)(C_1^2+ C_2^2)\neq0$. Then, the map $\varphi:(S^2, {\rm d}r^2+\sin^2 r{\rm d}\theta^2)\longrightarrow
(N^2,{\rm d}\rho^2+\lambda^2(\rho){\rm d}\phi^2)$ with $\varphi(r,\theta)=(\rho(r),k\theta)$ and $\rho(r)$ defined by $\rho(r)=(C_1-C_2+C_3)|\cot\frac{r}{2}|+(2C_1\ln |\tan\frac{r}{2}|+C_4)|\tan\frac{r}{2}|-(C_1|\tan\frac{r}{2}|+C_2|\cot\frac{r}{2}|)\ln(1+\tan^2\frac{r}{2})$ is a non-harmonic $f$-biharmonic map for $f=c_1+c_2\int (C_1\cot\frac{ r}{2}+C_2\tan \frac{r}{2})^{-2}\sin^{-1}r{\rm d}r$.
\end{theorem}
\begin{example}
Applying Theorem \ref{stcoy} with $c_1=c_2=C=C_2=C_4=k=1, C_0=1/2, C_1=C_3=0$, we can easily check that the map\\$ \varphi: S^2\supset((\pi/2,\pi)\times S^1, {\rm d}r^2+\sin^2r{\rm d}\theta^2)\longrightarrow
(N^2,{\rm d}\rho^2+(\rho+1){\rm d}\phi^2)$ with $\varphi(r,\theta)=(\frac{1}{4}(\ln\tan\frac{r}{2})^2-\ln\sin r+1,\theta)$ is a non-harmonic $f$-biharmonic map for $f=1+4\ln\tan\frac{r}{2}$ defined on an open subset of the standard $2$-sphere.\\
\end{example}
\begin{example} Similarly, applying Theorem \ref{stcoy1} with $c_1=1,\;c_2=-3,\;C_1=C_3=C_4=0, C_2=-1$, we obtain $f(r)=1+\frac{3}{2\sin^2\frac{r}{2}}$ and a non-harmonic $f$-biharmonic map $\varphi:(S^2, {\rm d}r^2+\sin^2r{\rm d}\theta^2)\longrightarrow
(N^2,{\rm d}\rho^2+(\rho^2+2C_0\rho+C){\rm d}\phi^2)$, $\varphi(r,\theta)=(|\cot\frac{r}{2}|[1+\ln(1+\tan^2\frac{r}{2})],\theta)$, which is defined on the standard $2$-sphere with two points ($r=0$ and $r=\pi$) deleted. Note that both the map and the function $f$ can be continuously extended over $r=\pi$.
\end{example}
Our next theorem shows that when the target surface is taken to be $(N^2,h={\rm d}\rho^2+\rho{\rm d}\phi^2)$ with positive Gauss curvature $K=\frac{1}{4\rho}$, the equation of rotationally symmetric biharmonic maps from the standard sphere reduces to a Cauchy-Euler equation, which can be solved by integrations.
\begin{theorem}\label{ps}
The rotationally symmetric map $\varphi:(S^2\setminus\{N,\;S\}, \frac{4({\rm d}r^2+r^2{\rm d}\theta^2)} {(1+r^2)^2} )\longrightarrow
(N^2, {\rm d}\rho^2+\rho{\rm d}\phi^2)$ with $\varphi(r,\theta)=(\rho(r), k\theta)$ is a  biharmonic map if and only if
\begin{eqnarray}
\rho(r)=&&C_1+ (C_2-2C_3)\ln r+C_3\ln (1+r^2)+C_4\ln r \ln (1+r^2)\\\notag&& -2C_4\int\frac{\ln(1+r^2)}{r}\,dr+\frac{k^2}{4}(\ln r)^2,
\end{eqnarray}
where $C_i, i=1, 2, 3, 4$ are constants. In particular, the rotationally symmetric map $\varphi:(S^2/\{N,\;S\}, \frac{ 4({\rm d}r^2+r^2{\rm d}\theta^2)}{(1+r^2)^2})\longrightarrow (N^2, {\rm d}\rho^2+\rho{\rm d}\phi^2)$ with $\varphi(r,\theta)=(\frac{k^2}{4}(\ln r)^2+\ln(1+r^2), k\theta)$ is a proper biharmonic map.
\end{theorem}
\begin{proof}
Applying Lemma \ref{sl} to the map $\varphi:(M^2,\bar{g}=f^{-1}(r)({\rm d}r^2+\sigma(r)^2{\rm d}\theta^2))\longrightarrow
(N^2, {\rm d}\rho^2+\lambda^2(\rho){\rm d}\phi^2)$, $\varphi(r,\theta)=(\rho(r), k\theta)$, with $\sigma=r>0, \lambda^2(\rho)=\rho, f=(1+r^2)^2/4$, we conclude that $\varphi$ is biharmonic if and only if
\begin{equation}
\begin{cases}\label{jbu1}
y''+\frac{1}{r}y'=0,\\
y=\frac{1}{4}(1+r^2)^2x,\\
x=\rho''+\frac{1}{r}\rho'-\frac{k^2}{2r^2}.
\end{cases}
\end{equation}
Since $y=\frac{1}{4}(1+r^2)^2x$ depends on $r$ alone,  the first equation in (\ref{jbu1}) is a second order linear homogeneous Cauchy-Euler equation whose general solution is given by $y=C_3+C_4\ln r$. It follows from the second equation of (\ref{jbu1}) that $x=4[C_3+C_4\ln r ](1+r^2)^{-2}$. Substituting this into the third equation of the (\ref{jbu1}) we have
\begin{equation}
\rho''+\frac{1}{r}\rho'=\frac{4[C_3+C_4\ln r ]}{(1+r^2)^{2}} +\frac{k^2}{2r^2}.
\end{equation}
This is a second order linear nonhomogeneous Cauchy-Euler equation whose general solution can be found by using the solution structure of linear DEs and the method of variation of parameters to be
\begin{eqnarray}
\rho(r)=&&C_1+ (C_2-2C_3)\ln r+C_3\ln (1+r^2)+C_4\ln r \ln (1+r^2)\\\notag&& -2C_4\int\frac{\ln(1+r^2)}{r}\,dr+\frac{k^2}{4}(\ln r)^2,
\end{eqnarray}
where $C_i, i=1, 2, 3, 4$ are constants. Thus, we obtain the first statement of the theorem. The second statement of the theorem follows from the particular solution corresponding to $C_1=C_4=0, C_3=1, C_2=2$.
\end{proof}
Similarly, we have 
\begin{proposition}\label{ps}
The rotationally symmetric map $\varphi:(S^2\setminus\{N,\;S\}, dr^2+\sin^2 r{\rm d}\theta^2)\longrightarrow
(N^2, {\rm d}\rho^2+\rho{\rm d}\phi^2)$ with $\varphi(r,\theta)=(\rho(r), k\theta)$ is an $f$-biharmonic map for $f(r)=\frac{\sin^2r}{[1+(\ln\tan \frac{r}{2})^2]^2}$ if and only if
\begin{eqnarray}\notag
\rho(r)=&&C_4 +C_3\ln \tan \frac{r}{2} +(\frac{C_1}{2} +\frac{k^2}{4})(\ln \tan \frac{r}{2})^2 +\frac{C_2}{6}(\ln \tan \frac{r}{2})^3+\frac{C_1}{6}(\ln \tan \frac{r}{2})^4\\\label{T9} &&+\frac{C_2}{10}(\ln \tan \frac{r}{2})^5+\frac{C_1}{30}(\ln \tan \frac{r}{2})^6+\frac{C_2}{42}(\ln \tan \frac{r}{2})^7,
\end{eqnarray}
where $C_i, i=1, 2, 3, 4$ are constants. In particular, the rotationally symmetric map $\varphi:(S^2\setminus\{N,\;S\}, dr^2+\sin^2 r{\rm d}\theta^2)\longrightarrow
(N^2, {\rm d}\rho^2+\rho{\rm d}\phi^2)$, $\varphi(r,\theta)=(\rho(r), k\theta)$ with $\rho(r)=\frac{1}{30}(\ln \tan \frac{r}{2})^6+\frac{1}{6}(\ln \tan \frac{r}{2})^4+(\frac{1}{2}+\frac{1}{4}k^2)(\ln \tan \frac{r}{2})^2 +1$, is a proper $f$-biharmonic map for $f(r)=\frac{\sin^2r}{[1+(\ln\tan \frac{r}{2})^2]^2}$.
\end{proposition}
\begin{proof}
We know that the rotationally symmetric map $\varphi:(S^2\setminus\{N,\;S\}, dr^2+\sin^2 r{\rm d}\theta^2)\longrightarrow
(N^2, {\rm d}\rho^2+\rho{\rm d}\phi^2)$ with $\varphi(r,\theta)=(\rho(r), k\theta)$ is a proper $f$-biharmonic map for $f(r)=\frac{\sin^2r}{[1+(\ln\tan \frac{r}{2})^2]^2}$ if and only if the map \\$\varphi:\left(S^2\setminus\{N,\;S\}, \frac{[1+(\ln\tan \frac{r}{2})^2]^2}{\sin^2r}(dr^2+\sin^2 r{\rm d}\theta^2)\right)\longrightarrow
(N^2, {\rm d}\rho^2+\rho{\rm d}\phi^2)$ with $\varphi(r,\theta)=(\rho(r), k\theta)$ is a biharmonic map. On the other hand, a straightforward computation shows that the Riemann sphere $\left((S^2\setminus\{N,\;S\}, \frac{[1+(\ln\tan \frac{r}{2})^2]^2}{\sin^2r}(dr^2+\sin^2 r{\rm d}\theta^2)\right)$ is isometric to the model $(\mathbb{R}^2\setminus\{0\}, (1+t^2)^2({\rm d}t^2+{\rm d}\theta^2))$ via the isometry $\begin{cases} t=\ln \tan \frac{r}{2} \\\theta=\theta\end{cases}.$ It follows that the map $\varphi:(S^2\setminus\{N,\;S\}, dr^2+\sin^2 r{\rm d}\theta^2)\longrightarrow
(N^2, h={\rm d}\rho^2+\rho{\rm d}\phi^2)$ with $\varphi(r,\theta)=(\rho(r), k\theta)$ is a proper $f$-biharmonic map for $f(r)=\frac{\sin^2r}{[1+(\ln\tan \frac{r}{2})^2]^2}$ if and only if the map $\varphi: (\mathbb{R}^2\setminus\{0\}, (1+t^2)^2({\rm d}t^2+{\rm d}\theta^2))\longrightarrow
(N^2, {\rm d}\rho^2+\rho{\rm d}\phi^2)$ with $\varphi(t,\theta)=(\rho(t), k\theta)$ is a biharmonic map, which is equivalent to stating that the map $\varphi: (\mathbb{R}^2\setminus\{0\}, {\rm d}t^2+{\rm d}\theta^2)\longrightarrow
(N^2, {\rm d}\rho^2+\rho{\rm d}\phi^2)$ with $\varphi(t,\theta)=(\rho(t), k\theta)$ is an $f$-biharmonic map with $f(t)=(1+t^2)^{-2}$.\\

Now applying Lemma \ref{sl} to the map $\varphi:(\mathbb{R}^2\setminus\{0\}, f^{-1}(t)({\rm d}t^2+{\rm d}\theta^2)\longrightarrow
(N^2, {\rm d}\rho^2+\lambda^2(\rho){\rm d}\phi^2)$, $\varphi(t,\theta)=(\rho(t), k\theta)$, with $\sigma=1>0, \lambda^2(\rho)=\rho, f(t)=(1+t^2)^{-2}$, we conclude that $\varphi$ is biharmonic if and only if
\begin{equation}
\begin{cases}\label{EC}
y''=0,\\
y=\frac{1}{(1+t^2)^2}x,\\
x=\rho''-\frac{k^2}{2}.
\end{cases}
\end{equation}
Solving the first equation in (\ref{EC}) we have $y=C_1+C_2t$ from which, together with the second equation of (\ref{EC}), we obtain that $x=(C_1+C_2 t )(1+t^2)^{2}$. Substituting this into the third equation of the (\ref{EC}) and integrating the resulting equation we have
\begin{equation}\label{90}
\rho(t)=C_4 +C_3t +(\frac{C_1}{2} +\frac{k^2}{4})t^2 +\frac{C_2}{6}t^3+\frac{C_1}{6}t^4+\frac{C_2}{10}t^5+\frac{C_1}{30}t^6+\frac{C_2}{42}t^7,
\end{equation}
where $C_i, i=1, 2, 3, 4$ are constants. Substituting $t=\ln \tan\frac{r}{2}$ into (\ref{90}) we obtain the first statement of the proposition. The second statement of the proposition follows from the particular solution corresponding to $C_1=C_4=1, C_2= C_3=0$.
\end{proof}

\section{Biharmonic and $f$-biharmonic maps from a round sphere $S^2$}
So far, no example of proper biharmonic map or $f$-biharmonic map from $S^2$ has been found except the biharmonic isometric immersion $\varphi: S^2\longrightarrow S^3$ with $\varphi (x)=(\frac{x}{\sqrt{2}}, \frac{1}{\sqrt{2}})$. In this section, we will show that in several special cases, $f$-biharmonic equation for rotationally symmetric maps can be solved to produce examples of proper $f$-biharmonic maps from a round sphere or proper biharmonic maps from Riemann spheres with some singular points.\\

As we mentioned in Section 2, the globally defined smooth map
$\varphi: S^2\longrightarrow S^2$ obtained from the restriction of
the polynomial map $F:\mathbb{R}^{3}\longrightarrow \mathbb{R}^{3}$
with $F(x, y, z)=(x^{2}-y^{2}-z^{2}, \;2xy,\; 2xz)$ is a rotationally symmetric map
\begin{eqnarray}\notag
&&\varphi :((0,\pi/2)\times S^{1}, {\rm d}\,r^{2}+\sin ^{2} r\,{\rm d}\,\theta^{2})\longrightarrow ((0,\pi)\times S^{1}, {\rm d}\,\rho^{2}+\sin ^{2} \rho\,{\rm d}\,\phi^{2}),\\ \notag
&&\varphi (r,\theta)=(2r, \theta).
\end{eqnarray}
Ou next proposition shows that the $f$-biharmonicity of this map reduces to a Riccati equation satisfied by the derivative of $\beta= \frac{1}{2} \ln f$.
\begin{proposition}\label{tcll}
The rotationally symmetric map $\varphi:(S^2, dr^2+\sin^2 rd\theta^2)\longrightarrow
(S^2, d\rho^2+\sin^2 \rho d\phi^2)$ with $\varphi(r,\theta)=(2r,\theta)$
is an $f$-biharmonic map if and only if $\beta=\frac{1}{2}\ln f(r)$ solves the following equation
\begin{equation}\label{b1}
\beta''+(3\cot r-2\tan r)\beta'+2 \beta'^2+1-4\sin^2r=0,
\end{equation}
which is a Riccati equation in $\beta'$.
\end{proposition}
\begin{proof}
By Corollary \ref{f-bih}, $\varphi:(S^2, dr^2+\sin^2 rd\theta^2)\longrightarrow
(S^2, d\rho^2+\sin^2 \rho d\phi^2)$ with $\varphi(r,\theta)=(2r,\theta)$
is an $f$-biharmonic map if and only if 
\begin{equation}\label{2r}
\begin{cases}
y''+ \frac{\cos r}{\sin r} y'-\frac{\cos4r}{\sin^2r}y=0,\\
y=fx,\\
x=2\sin2r.
\end{cases}
\end{equation}
By setting $\beta=\frac{1}{2}\ln f$ we can check that Equation (\ref{2r}) is equivalent to
\begin{equation}\label{btcll2}
\beta''+(3\cot r-2\tan r)\beta'+2 \beta'^2+1-4\sin^2r=0.
\end{equation}
Thus, the proposition follows.
\end{proof}
\begin{remark}
We know from \cite{WOY} that the rotationally symmetric map
$\varphi:(S^2, dr^2+\sin^2 r d\theta^2)\longrightarrow
(S^2, d\rho^2+\sin^2 \rho d\phi^2)$ with
$\varphi(r,\theta)=(2r,\theta)$ is neither harmonic nor biharmonic. However, by Theorem \ref{tcll} and the existence of solution of Riccati equation, we know that there are functions $f$ locally defined on $S^2$ so that the map
$\varphi:(S^2\supset U, f^{-1}(dr^2+\sin^2 r d\theta^2))\longrightarrow
(S^2, d\rho^2+\sin^2 \rho d\phi^2)$ with
$\varphi(r,\theta)=(2r,\theta)$ is proper biharmonic map. In other words, the existence of solutions of Riccati equation and our Theorem \ref{tcll} provides many locally defined $f$-biharmonic maps between standard $2$-spheres. 
\end{remark}
Our next theorem gives an $f$-biharmonic map defined on $S^2\setminus\{N, S\}$ which can be extended continuously to $S^2$.
\begin{theorem}\label{glob}
For $f(r)
= \frac{4(
1+\tan^2\frac{r}{2}) (1+2\tan^2\frac{r}{2})^{\frac{3}{2}}}{
3\tan^4\frac{r}{2}+9\tan^2\frac{r}{2}+6+\cot^2\frac{r}{2}}$, the map $\varphi:(S^2\setminus\{N, S\}, f^{-1}(r)(dr^2+\sin^2 r d\theta^2))\longrightarrow
(S^2, d\rho^2+\sin^2 \rho d\phi^2)$ with $\varphi(r,\theta)=(\frac{1}{2}\arccos(\sin^{2}\frac{r}{2}), 2\theta)$ is proper biharmonic. In other words, the map $\varphi:(S^2, dr^2+\sin^2 r d\theta^2)\longrightarrow
(S^2, d\rho^2+\sin^2 \rho d\phi^2)$ with $\varphi(r,\theta)=(\frac{1}{2}\arccos(\sin^{2}\frac{r}{2}), 2\theta)$ between the standard spheres is a proper $f$-biharmonic map for $f(r)
= \frac{4(
1+\tan^2\frac{r}{2}) (1+2\tan^2\frac{r}{2})^{\frac{3}{2}}}{
3\tan^4\frac{r}{2}+9\tan^2\frac{r}{2}+6+\cot^2\frac{r}{2}}$.
\end{theorem}
\begin{proof}
By Lemma \ref{sl}, the map $\varphi:( S^2,\bar{g}=f^{-1}(r)(dr^2+\sin^2 r d\theta^2))\longrightarrow
(S^2, d\rho^2+\sin^2 \rho d\phi^2)$ with $\varphi(r,\theta)=(\frac{1}{2}\arccos(\sin^{2}\frac{r}{2}), 2\theta)$, $ \sigma=\sin r$, and $\lambda(\rho)=\sin\rho$ is biharmonic if and only if
\begin{equation}\label{gl1}
\begin{cases}
y''+\cot ry'-\frac{4\sin^{2}\frac{r}{2}}{\sin^2r}y=0,\\
y=f x,\\
x=-\frac{6+3\tan^4\frac{r}{2}+9\tan^2\frac{r}{2}+\cot^2\frac{r}{2}}{2(1+2\tan^2\frac{r}{2})^{\frac{3}{2}}}.
\end{cases}
\end{equation}
Using substitution $t=\ln|\tan\frac{r}{2}|$, we can re-write Equation (\ref{gl1}) as
\begin{equation}\label{gl2}
\begin{cases}
\frac{d^2y}{d t^2}-4 \frac{e^{2t}}{1+e^{2t}}y(t)=0,\\
y(t)=f(t)x(t),\\
x(t)= -\frac{3e^{4t}+9e^{2t}+6+e^{-2t}}{2(1+2e^{2t})^{\frac{3}{2}}}.
\end{cases}
\end{equation}

It is easy to check that $y_1=1+e^{2t}$ is a solution of the 2nd order linear DE (\ref{gl2}). Using the method of reduction of order we have the general solution of the first equation of (\ref{gl2}) given by
\begin{equation}\label{gl3}
y(t)=(1+e^{2t})\left\{C_1\left(\frac{1}{2(1+e^{2t})}+\frac{1}{2}\ln\frac{e^{2t}}{1+e^{2t}}\right)+C_2\right\},
\end{equation}
where $C_1, C_2$ are constants.\\

By choosing $C_1=0, C_2=-2$ we have a special solution $y_0(t)=-2(1+e^{2t})$. It follows that we can solve for $f$ from the second equation of (\ref{gl2}) to have
\begin{eqnarray}\notag
f(t)=y_0(t)/x(t)= \frac{4(1+e^{2t})(1+2e^{2t})^{\frac{3}{2}}}{(3e^{4t}+9e^{2t}+6+e^{-2t})}.\\\notag
\end{eqnarray}
Using the substitution $t=\ln|\tan\frac{r}{2}|$ we have
\begin{eqnarray}\notag
f(r)
&=& \frac{4(
1+\tan^2\frac{r}{2}) (1+2\tan^2\frac{r}{2})^{\frac{3}{2}}}{
3\tan^4\frac{r}{2}+9\tan^2\frac{r}{2}+6+\cot^2\frac{r}{2}},
\end{eqnarray}
and that the map $\varphi:(S^2,f^{-1}(r)(dr^2+\sin^2 r d\theta^2))\longrightarrow
(S^2, d\rho^2+\sin^2 \rho d\phi^2)$ defined by $\varphi(r,\theta)=(\frac{1}{2}\arccos(\sin^{2}\frac{r}{2}), 2\theta)$ is proper biharmonic.
Thus, the theorem follows.
\end{proof}
\begin{remark}
We can check that
\begin{equation}
\begin{cases}
\lim\limits_{r\longrightarrow 0}\rho(r)=\lim\limits_{r\longrightarrow 0}\frac{1}{2}\arccos(\sin^{2}\frac{r}{2})=\frac{\pi}{4}, \\ \lim\limits_{r\longrightarrow \pi}\rho(r)=\lim\limits_{r\longrightarrow \pi}\frac{1}{2}\arccos(\sin^{2}\frac{r}{2})=0,\\
\lim\limits_{r\longrightarrow 0}f^{-1}(r)=\lim\limits_{r\longrightarrow 0}\frac{
3\tan^4\frac{r}{2}+9\tan^2\frac{r}{2}+6+\cot^2\frac{r}{2}}{4(
1+\tan^2\frac{r}{2}) (1+2\tan^2\frac{r}{2})^{\frac{3}{2}}}=+\infty,\\
\lim\limits_{r\longrightarrow \pi}f^{-1}(r)=\lim\limits_{r\longrightarrow \pi}\frac{
3\tan^4\frac{r}{2}+9\tan^2\frac{r}{2}+6+\cot^2\frac{r}{2}}{4(
1+\tan^2\frac{r}{2}) (1+2\tan^2\frac{r}{2})^{\frac{3}{2}}}=0.
\end{cases}
\end{equation}

It follows that the function $\rho: (0, \pi)\longrightarrow (0, \pi)$ can be continuously extended over $r=0, \pi$, so the map $\varphi$ can be continuously extended to the whole sphere $S^2$. This can also be seen from the expression of the map: $\varphi: (\sin r \cos \theta, \sin r \sin \theta, \cos r) \mapsto \left(\frac{\cos \frac{r}{2}}{\sqrt{2}} \cos 2\theta, \frac{\cos \frac{r}{2}}{\sqrt{2}} \sin 2\theta, \frac{\sqrt{1+\sin^2 \frac{r}{2}}}{\sqrt{2}} \right)$. However, the Riemann sphere $(S^2, f^{-1}g_0)$ ($g_0$ is the standard metric on $S^2$) still has singular points at two poles since the function $f^{-1}$ does not allow a continuously extension over $r=0$ and the continuous extension of $f^{-1}$ over $r=\pi$ has zero value at $r=\pi$. So, the $f$-biharmonic map $\varphi:(S^2\setminus\{N, S\}, (dr^2+\sin^2 r d\theta^2))\longrightarrow
(S^2, d\rho^2+\sin^2 \rho d\phi^2)$ provided by Theorem \ref{glob} can be continuously extended over the north and the south poles. 
\end{remark}
\begin{lemma}\label{lfpyj}
For nonzero constants $A$ and $k$, the rotationally symmetric map $\varphi:( S^2, dr^2+\sin^2 r d\theta^2)\longrightarrow
( S^2, d\rho^2+\sin^2 \rho d\phi^2)$ with $\varphi(r,\theta)=(Ar, k\theta)$ and $\tau(\varphi)\ne 0$ is an $f$-biharmonic map if and only if 
$f(r, \theta)=\frac{y(\ln|\tan \frac{r}{2}|)}{A\cot r-\frac{k^2\sin(2Ar)}{2\sin^2r}}$, where $y(t)$ is a solution of the following ODE
\begin{equation}\label{lfpyj0}
\frac{d^2y}{d t^2}-k^2\cos(4A\arctan e^t)y(t)=0.
\end{equation}
\end{lemma}
\begin{proof}
Recall that the map $\varphi:(S^2, g= dr^2+\sin^2r d\theta^2)\longrightarrow
(S^2, d\rho^2+\sin^2\rho d\phi^2)$ with $\varphi(r,\theta)=(Ar, k\theta)$ is an $f$-biharmonic map if and only if
the map $\varphi:(S^2, f^{-1}g)\longrightarrow
(S^2, d\rho^2+\sin^2\rho d\phi^2)$ is biharmonic from a Riemann sphere. This, together with Lemma \ref{sl} and $\tau(\varphi)\ne 0$, implies that $f$ depends only on variable $r$, and
\begin{equation}\label{lfpyj1}
\begin{cases}
y''+\cot ry'-\frac{k^2\cos(2Ar)}{\sin^2r}y=0,\\
y=fx,\\
x=A\cot r-\frac{k^2\sin(2Ar)}{2\sin^2r}.
\end{cases}
\end{equation}
A straightforward computation shows that the substitution $t=\ln|\tan \frac{r}{2}|$ turns the first equation of (\ref{lfpyj1}) into
\begin{equation}\label{lfpyj2}
\frac{d^2y}{d t^2}-k^2\cos(4A\arctan e^t)y(t)=0.
\end{equation}
From the second equation of (\ref{lfpyj1}) we have
\begin{equation}
f=\frac{y(t)}{A\cot r-\frac{k^2\sin(2Ar)}{2\sin^2r}}=\frac{y(\ln|\tan \frac{r}{2}|)}{A\cot r-\frac{k^2\sin(2Ar)}{2\sin^2r}},
\end{equation}
where $y(t)$ is a solution of Equation (\ref{lfpyj2}).
From this we obtain the lemma.
\end{proof}
\begin{theorem}\label{kzt}
The rotationally symmetric map $\varphi:( S^2, dr^2+\sin^2r d\theta^2)\longrightarrow
( S^2, d\rho^2+\sin^2 \rho d\phi^2)$ with $\varphi(r,\theta)=(r, \sqrt{3}\theta)$ is an $f$-biharmonic map for
$f(r, \theta)=|\frac{(\tan\frac{r}{2})^{1+\sqrt{3}}(\tan^{2}\frac{r}{2}+7+4\sqrt{3})(
\tan^{2}\frac{r}{2}-2-\sqrt{3})
}{(1+\tan^{2}\frac{r}{2})^2(\tan^{2}\frac{r}{2}-1)}|$, which is smooth on $(0, \pi)$ except $r=\frac{\pi}{2}, 2\arctan \sqrt{2+\sqrt{3}\,}$.
\end{theorem}
\begin{proof}
Applying Lemma \ref{lfpyj} with $ A=1, k=\sqrt{3}$ we conclude that the map $\varphi:( S^2, dr^2+\sin^2r d\theta^2)\longrightarrow
( S^2, d\rho^2+\sin^2 \rho d\phi^2)$ with $\varphi(r,\theta)=(r, \sqrt{3}\theta)$ is an $f$-biharmonic map if and only if 
\begin{equation}\label{GD1}
f(r, \theta)=\frac{y(t)}{(1-k^2)\cot r}=-\frac{1}{2}y(\ln|\tan \frac{r}{2}|)\tan r,
\end{equation}
and $y(t)$ is a solution of the following equation
\begin{equation}\label{kzt2}
\frac{d^2y}{d t^2}-3\cos(4\arctan e^t)y(t)=0,
\end{equation}
which is equivalent to
\begin{equation}\label{kzt3}
\frac{d^2y}{d t^2}-3\frac{e^{4t}-6e^{2t}+1}{(1+e^{2t})^2}y(t)=0.
\end{equation}
To solve Equation (\ref{kzt3}), we look for a particular solution of the form $y(t)=e^{\sqrt{3}t}(1+e^{2t})^{-2}\sum\limits_{j=0}^{2}a_je^{2jt}$ with coefficients $a_j,\; j=0, 1, 2$ to be determined. Substituting $y(t)$ into (\ref{kzt3}) we have 
\begin{equation}\label{kzt7}
\begin{cases}
2(2-\sqrt{3})a_0+(\sqrt{3}+1)a_1=0,\\
2(2-\sqrt{3})a_0+2a_1+2(2+\sqrt{3})a_2=0,\\
(1-\sqrt{3})a_{1}+2(2+\sqrt{3})a_{2}=0.
\end{cases}
\end{equation}
Solving this system of linear equations we obtain $a_0=-(26+15\sqrt{3})a_2,\; a_1=(5+3\sqrt{3})a_2$. In particular, choosing $a_2=1$, we have $a_0=-26-15\sqrt{3},\; \;a_1=5+3\sqrt{3}$, which gives a particular solution of (\ref{kzt3}) as 
\begin{equation}\label{kzt9}
y(t)=\frac{e^{\sqrt{3}t}[-26-15\sqrt{3} +(5+3\sqrt{3})e^{2t}+ e^{4t} ]}{(1+e^{2t})^2}.
\end{equation}
Substituting this into (\ref{GD1}), we have
\begin{eqnarray}\label{kzt11}\notag
f(r)&=&-\frac{1}{2}y(\ln|\tan \frac{r}{2}|)\tan r\\\notag
&=&\frac{(\tan\frac{r}{2})^{1+\sqrt{3}}[-26-15\sqrt{3} +(5+3\sqrt{3})\tan^{2}\frac{r}{2}+ \tan^{4}\frac{r}{2} ]
}{(1+\tan^{2}\frac{r}{2})^2(\tan^{2}\frac{r}{2}-1)}\\\label{GD2}
&=&\frac{(\tan\frac{r}{2})^{1+\sqrt{3}}(\tan^{2}\frac{r}{2}+7+4\sqrt{3})(
\tan^{2}\frac{r}{2}-2-\sqrt{3})
}{(1+\tan^{2}\frac{r}{2})^2(\tan^{2}\frac{r}{2}-1)}.
\end{eqnarray}

It is easy to check from (\ref{GD2}) that (i) $f$ is smooth on $(0, \pi)$ except $r=\frac{\pi}{2}$, (ii) $r=2\arctan \sqrt{2+\sqrt{3}}$ is the only zero of $f$ within $( 0, \pi)$, and (iii) $ f(r)>0, \forall \, r\in (0, \frac{\pi}{2})\cup (2\arctan \sqrt{2+\sqrt{3}}, \pi)$ and $ f(r)<0, \forall \, r\in (\frac{\pi}{2}, 2\arctan \sqrt{2+\sqrt{3}})$. Note that the definition of $f$-biharmonic map (or $f^{-1}$ being a conformal factor) requires that $f>0$, so the function $f$ given by (\ref{GD2}) is inappropriate on the interval $(\frac{\pi}{2}, 2\arctan \sqrt{2+\sqrt{3}})$. However, we can use $f(r)= -\frac{(\tan\frac{r}{2})^{1+\sqrt{3}}(\tan^{2}\frac{r}{2}+7+4\sqrt{3})(
\tan^{2}\frac{r}{2}-2-\sqrt{3})
}{(1+\tan^{2}\frac{r}{2})^2(\tan^{2}\frac{r}{2}-1)}$ for the interval $(\frac{\pi}{2}, 2\arctan \sqrt{2+\sqrt{3}})$. This is possible because $f$ is a function such that $y=xf$ is a solution of a 2nd order linear differential equation and the fact that if $y=xf$ is a solution of a linear differential equation, then so is $y=-xf$. From this and the above mentioned properties of $f$ we conclude that $f(r, \theta)=|\frac{(\tan\frac{r}{2})^{1+\sqrt{3}}(\tan^{2}\frac{r}{2}+7+4\sqrt{3})(
\tan^{2}\frac{r}{2}-2-\sqrt{3})
}{(1+\tan^{2}\frac{r}{2})^2(\tan^{2}\frac{r}{2}-1)}|$ is smooth on $(0, \pi)$ except $r=\frac{\pi}{2}, 2\arctan \sqrt{2+\sqrt{3}\,}$. This completes the proof of the theorem.
\end{proof}

In a similar way, we have
\begin{theorem}\label{G3}
The rotationally symmetric map $\varphi:( S^2, dr^2+\sin^2r d\theta^2)\longrightarrow
( S^2, d\rho^2+\sin^2 \rho d\phi^2)$ with $\varphi(r,\theta)=(r, -\sqrt{3}\theta)$ is a proper $f$-biharmonic map for
$f(r, \theta)=|\frac{(\tan\frac{r}{2})^{1-\sqrt{3}}(\tan^{2}\frac{r}{2}+7-4\sqrt{3})(
\tan^{2}\frac{r}{2}-2+\sqrt{3})
}{(1+\tan^{2}\frac{r}{2})^2(\tan^{2}\frac{r}{2}-1)}|$, which is smooth on $(0, \pi)$ except $r=\frac{\pi}{2}, 2\arctan \sqrt{2-\sqrt{3}\,}$.
\end{theorem}
\begin{theorem}\label{kztt}
The rotationally symmetric map $\varphi:( S^2, dr^2+\sin^2r d\theta^2)\longrightarrow
( S^2, d\rho^2+\sin^2 \rho d\phi^2)$ with $\varphi(r,\theta)=(r, \frac{1}{2}\theta)$ is an $f$-biharmonic map if and only if
$f(r, \theta)=|\frac{\tan r\sin(\frac{\sqrt{3}}{2} r) }{\sqrt{\sin
r}}|$, which is smooth on $(0, \pi)$ except $r=\frac{\pi}{2}$.
\end{theorem}
\begin{proof}
Applying Lemma \ref{lfpyj} with $ A=1, k^2=\frac{1}{4}$ we conclude that, the maps $\varphi:( S^2, dr^2+\sin^2r d\theta^2)\longrightarrow
( S^2, d\rho^2+\sin^2 \rho d\phi^2)$ with $\varphi(r,\theta)=(r, \pm \frac{1}{2}\theta)$ is an $f$-biharmonic map if and only if 
\begin{equation}\label{kztt1}
\frac{d^2y}{d t^2}-\frac{1}{4}\frac{e^{4t}-6e^{2t}+1}{(1+e^{2t})^2}y(t)=0.
\end{equation}
To solve Equation (\ref{kztt1}) we look for particular solutions of the form 
$y(t)=u(t)e^{iv(t)}$. Substituting this into (\ref{kztt1})
we have 
\begin{equation}\label{kztt2}
\begin{cases}
(u''-(v'^2+\frac{1}{4}\frac{e^{4t}-6e^{2t}+1}{(1+e^{2t})^2})u=0,\\
v''u+2v'u'=0.
\end{cases}
\end{equation}
Solving the second equation of (\ref{kztt2}), we have
\begin{equation}\label{kztt4}
v'=C/u^2,
\end{equation}
where $C$ is a constant.\\

Taking $u=\sqrt{\frac{1+e^{2t}}{e^t}}$ and $C=\sqrt{3}$ we check that $v'=\frac{\sqrt{3}\,e^t}{1+e^{2t}}$, 
and hence $v=\frac{\sqrt{3}}{2}\arctan(\frac{2e^t}{1-e^{2t}})$. Another straightforward checking shows that 
$u=\sqrt{\frac{1+e^{2t}}{e^t}}$ and $v=\frac{\sqrt{3}}{2}\arctan(\frac{2e^t}{1-e^{2t}})$ indeed solve
the first equation of
(\ref{kztt2}). From this, we obtain the general solutions of (\ref{kztt1}) as
\begin{equation}\label{kztt6}
y(t)=\sqrt{\frac{1+e^{2t}}{e^t}}\{C_1\sin\left(\frac{\sqrt{3}}{2}\arctan(\frac{2e^t}{1-e^{2t}})\right)
+C_2\cos\left(\frac{\sqrt{3}}{2}\arctan(\frac{2e^t}{1-e^{2t}})\right),
\end{equation}
where $C_1, C_2$ are constants.\\

Taking $C_1=3/4\sqrt{2},C_2=0$, we obtain a special solution as
\begin{equation}\label{kztt7}
y(t)=3\sqrt{\frac{1+e^{2t}}{e^t}}\sin\left(\frac{\sqrt{3}}{2}\arctan(\frac{2e^t}{1-e^{2t}})\right)/4\sqrt{2}.
\end{equation}
Using Lemma \ref{lfpyj} with $ A=1, k^2=\frac{1}{4}$ we obtain 
\begin{equation}\label{kzt1}
f(r)=\frac{y(t)}{(1-k^2)\cot r}=\frac{y(\ln|\tan
\frac{r}{2}|)}{\frac{3}{4}}\tan r=\pm\frac{\tan
r \sin(\frac{\sqrt{3}}{2} r)}{\sqrt{\sin
r}}.
\end{equation}

An analysis similar to the one given in the proof of Theorem \ref{kzt} applies to complete the proof of the theorem.
\end{proof}
\begin{remark}
We remark that all maps studied in Theorems \ref{kzt}, \ref{G3} and \ref{kztt} are globally defined smooth maps between spheres. However, as $f$-biharmonic maps, each of them has some singular points. In other words, they are biharmonic maps $(S^2, f^{-1}g_0)\longrightarrow (S^2, g_0)$ from Riemann spheres with some singularities.
\end{remark}
Finally, we close the paper with the following corollary that provides examples of proper biharmonic maps from Riemann sphere into $S^3$.
\begin{corollary}\label{CLast}
let $\varphi_k: S^2\longrightarrow S^2, i=1, 2, 3, 4$ be the $f_k$-biharmonic maps defined in Theorems \ref{glob}, \ref{kzt}, \ref{G3} and \ref{kztt}, and let ${\bf I}: S^2\longrightarrow S^n$ ($n\ge 3$) be the standard totally geodesic embedding. Then, the map ${\bf I}\circ \varphi_k: (S^2, f_k^{-1}g_0)\longrightarrow S^n$ is a proper biharmonic map from Riemann sphere with some singular points for $k=1, 2, 3, 4$. None of such maps is weakly conformal.
\end{corollary}
\begin{proof}
It follows from Corollary \ref{C20} and Theorems \ref{glob}, \ref{kzt}, \ref{G3} and \ref{kztt} that the map $\varphi_k: (S^2, f_k^{-1}g_0)\longrightarrow S^2, i=1, 2, 3, 4$ is a proper biharmonic map. The corollary follows from this and the fact (See e.g., \cite{Ou1}) that the composition of a proper biharmonic map followed by a totally geodesic map is again a proper biharmonic map.
\end{proof}
It is interesting to compare the examples of proper biharmonic maps provided by Collorary \ref{CLast} with the Sack-Uhlenbeck's well-known theorem stating that any harmonic map from $S^2$ is conformal immersion away from points where the differential of the map vanishes.

\end{document}